\documentclass[a4paper,12pt]{article}
\usepackage{latexsym}
\usepackage{amsmath}
\usepackage{amssymb}
\usepackage{enumerate}

\usepackage{theorem}
\newtheorem{theorem}{Theorem}[section]
\newtheorem{proposition}[theorem]{Proposition}
\newtheorem{lemma}[theorem]{Lemma}
\newtheorem{corollary}[theorem]{Corollary}

\theorembodyfont{\rmfamily}
\newtheorem{proof}{\textmd{\textit{Proof.}}}

\newtheorem{remark}[theorem]{Remark}

\newtheorem{definition}[theorem]{Definition}

\makeatletter

\@addtoreset{equation}{section}
\makeatother

\newcommand{\qedd}{\hfill \Box}
\newcommand{\ve}{\varepsilon}
\newcommand{\lra}{\longrightarrow}
\newcommand{\wt}{\widetilde}
\newcommand{\wh}{\widehat}

\newcommand{\N}{\ensuremath{\mathbb{N}}}
\newcommand{\R}{\ensuremath{\mathbb{R}}}

\newcommand{\cD}{\ensuremath{\mathcal{D}}}

\newcommand{\cI}{\ensuremath{\mathcal{I}}}

\newcommand{\cN}{\ensuremath{\mathcal{N}}}
\newcommand{\cO}{\ensuremath{\mathcal{O}}}

\newcommand{\cU}{\ensuremath{\mathcal{U}}}

\def\inj{\mathop{\mathrm{inj}}\nolimits}

\def\ra{\mathop{\mathrm{rank}}\nolimits}

\def\Cut{\mathop{\mathrm{Cut}}\nolimits}
\def\Focal{\mathop{\mathrm{Foc}}\nolimits}

\title{Applications of Toponogov's comparison theorems for open triangles
}

\author{Kei KONDO \ $\cdot$ \ Minoru TANAKA}
\date{}
\begin{document}
\maketitle

\begin{abstract}
Recently we generalized Toponogov's comparison theorem to 
a complete Riemannian manifold with smooth convex boundary, where  
a geodesic triangle was replaced by an open (geodesic) triangle standing on the boundary 
of the manifold, and a model surface was replaced by the universal covering surface of 
a cylinder of revolution with totally geodesic boundary. 
The aim of this article is to prove splitting theorems of two types as an application. 
Moreover, we establish a weaker version of our Toponogov comparison theorem 
for open triangles, because the weaker version is quite enough to prove one of the splitting theorems.
\end{abstract}

\section{Introduction}\label{sec:int}
Words have fully expressed a matter of great importance for Toponogov's comparison theorem. 
However that may be, we can not stop telling the importance in Riemannian geometry. 
The comparison theorem has played a vital role in the comparison geometry, that is, 
the theorem gives us some techniques originating from Euclidean geometry. 
Such techniques, drawing a circle or a geodesic polygon, 
and joining two points by a minimal geodesic segment, 
are very powerful in the geometry. One may find concrete examples of such techniques in proofs of 
the maximal diameter theorem and the splitting theorem by Toponogov (\cite{T1}, \cite{T2}), 
the structure theorem with positive sectional curvature by Gromoll and Meyer (\cite{GM}), 
the soul theorem with non-negative sectional curvature by Cheeger and Gromoll (\cite{CG}), 
the diameter sphere theorem by Grove and Shiohama (\cite{GS}), etc.\par 
From the standpoint of the radial curvature geometry, 
we very recently generalized the Toponogov comparison theorem to 
a complete Riemannian manifold with smooth convex boundary, where  
a geodesic triangle was replaced by an open (geodesic) triangle standing on the boundary 
of the manifold, and a model surface was replaced by the universal covering surface of 
a cylinder of revolution with totally geodesic boundary (\cite[Theorem 8.4]{KT2}, which will be stated 
as Theorem \ref{thm4.9} in this article).

\bigskip

The aim of our article is to prove splitting theorems of two types as an application of 
Toponogov's comparison theorem for open triangles and a weaker version of the comparison theorem (Theorem \ref{weak}), respectively. The weaker version will be proved in this article. 

\bigskip

Now we will introduce the radial curvature geometry for manifolds with boundary: 
We first introduce our model, which will be later employed as a reference surface of 
comparison theorems in complete Riemannian manifolds with boundary. 
Let $\wt{M} := (\R, d\tilde{x}^2) \times_{m} (\R, d\tilde{y}^2)$
be a warped product of two $1$-dimensional Euclidean lines $(\R, d\tilde{x}^2)$ 
and $(\R, d\tilde{y}^2)$, where the warping function $m : \R \lra (0, \infty)$ is a positive 
smooth function satisfying $m(0) = 1$ and $m'(0) = 0$. Then we call 
\[
\wt{X} := \left\{ \tilde{p} \in \wt{M} \ | \ \tilde{x}(\tilde{p}) \ge 0 \right\} 
\]
a {\em model surface}. 
Since $m'(0) = 0$, the boundary 
$\partial \wt{X}:= \{ \tilde{p} \in \wt{X}\, |\, \tilde{x}(\tilde{p}) = 0 \}$ of 
$\wt{X}$ is {\em totally geodesic}. 
The metric $\tilde{g}$ of $\wt{X}$ is expressed as 
\begin{equation}\label{model-metric}
\tilde{g} = d\tilde{x}^2 + m(\tilde{x})^{2} d\tilde{y}^2
\end{equation}
on $[0, \infty) \times \R$. 
The function $G \circ \tilde{\mu} : [0,\infty) \lra \R$ is called the 
{\em radial curvature function} of $\wt{X}$, 
where we denote by $G$ the Gaussian curvature of $\wt{X}$, 
and by $\tilde{\mu}$ any ray emanating perpendicularly from $\partial \wt{X}$ 
(Notice that such a $\tilde{\mu}$ will be called a $\partial \wt{X}$-ray). 
Remark that $m : [0, \infty) \lra \R$ satisfies the differential equation 
$m''(t) + G (\tilde{\mu}(t)) m(t) = 0$ 
with initial conditions $m(0) = 1$ and $m'(0) = 0$.  
Note that the $n$-dimensional model surfaces 
are defined similarly, and, as seen in \cite{KK}, 
we may completely classify them by taking half spaces of spaces in \cite[Theorem 1.1]{MS}.\par
Hereafter, let $(X, \partial X)$ denote a complete Riemannian 
$n$-dimensional manifold $X$ with smooth boundary $\partial X$. 
We say that $\partial X$ is {\em convex}, 
if all eigenvalues of the shape operator $A_{\xi}$ of $\partial X$ are 
non-negative in the inward vector $\xi$ normal to $\partial X$. 
Notice that our sign of $A_{\xi}$ differs from \cite{S}. 
That is, for each $p \in \partial X$ and $v \in T_{p}\partial X$, 
$A_{\xi}(v) = -\,(\nabla_{v} N)^{\top}$ holds. 
Here, we denote by $N$ a local extension of $\xi$, 
and by $\nabla$ the Riemannian connection on $X$.\par 
For a positive constant $\ell$, 
a unit speed geodesic segment $\mu : [0, \ell] \lra X$ emanating from $\partial X$ 
is called a {\em $\partial X$-segment}, if 
$d(\partial X, \mu(t)) = t$ on $[0, \ell]$. 
If $\mu : [0, \ell] \lra X$ is a $\partial X$-segment for all $\ell > 0$, 
we call $\mu$ a {\em $\partial X$-ray}. 
Here, we denote by $d(\partial X, \, \cdot \, )$ the distance function to $\partial X$ 
induced from the Riemannian structure of $X$. 
Notice that a $\partial X$-segment 
is orthogonal to $\partial X$ by the first variation formula, and so a $\partial X$-ray is too.\par 
$(X, \partial X)$ is said to have the 
{\em 
radial curvature (with respect to $\partial X$) 
bounded from below by that of $(\wt{X}, \partial \wt{X})$
} 
if, for every $\partial X$-segment $\mu : [0, \ell) \lra X$, 
the sectional curvature $K_{X}$ of $X$ satisfies
\[
K_{X}(\sigma_{t}) \ge G (\tilde{\mu}(t))
\]
for all $t \in [0, \ell)$ and all $2$-dimensional linear spaces $\sigma_{t}$ spanned by $\mu'(t)$ 
and a tangent vector to $X$ at $\mu(t)$. 
For example, if the Riemannian metric of $\wt{X}$ is 
$d\tilde{x}^{2} + d\tilde{y}^{2}$, or $d\tilde{x}^2 + \cosh^{2} (\tilde{x})\,d\tilde{y}^{2}$, then 
$G (\tilde{\mu}(t)) = 0$, or $G (\tilde{\mu}(t)) = -1$, respectively. 
Furthermore, {\bf the radial curvature may change signs wildly}. Examples of a model surfaces 
admitting such a crazy behavior of radial curvature are found in \cite[Theorems 1.3 and 4.1]{TK}.

\bigskip

Our main theorems in this article are now stated as follows:

\begin{theorem}\label{thm1.4} 
Let $(X, \partial X)$ be a complete non-compact connected Riemannian manifold $X$ 
with smooth convex boundary $\partial X$ whose radial curvature is bounded 
from below by that of a model surface $(\wt{X}, \partial \wt{X})$ 
with its metric (\ref{model-metric}). 
Assume that $X$ admits at least one $\partial X$-ray. 
\begin{enumerate}[{\rm ({ST--}1)}]
\item
If $(\wt{X}, \partial \wt{X})$ satisfies 
\[
\int_{0}^{\infty} \frac{1}{m (t)^{2}}dt = \infty,
\]
then $X$ is isometric to $[0, \infty ) \times_{m} \partial X$. 
In particular, $\partial X$ is the soul of $X$, and 
the number of connected components of $\partial X$ is one.
\item
If $(\wt{X}, \partial \wt{X})$ satisfies $\liminf_{t \to \infty} m (t) = 0$, 
then $X$ is diffeomorphic to $[0, \infty ) \times \partial X$. 
In particular, the number of connected components of $\partial X$ is one.
\end{enumerate} 
\end{theorem}

\medskip\noindent 
Toponogov's comparison theorem for open triangles 
in a weak form (Theorem \ref{weak}) will be applied 
in the proof of Theorem \ref{thm1.4} (see Section \ref{sec:app1}). 
The assumption on the existence of a $\partial X$-ray is very natural, 
because we may find at least one $\partial X$-ray if $\partial X$ is compact. 
If the model $\wt{X}$ is Euclidean (i.e., $m \equiv 1$), then the (ST--1) holds. 
Hence, Theorem \ref{thm1.4} extends one of Burago and Zalgaller' splitting theorems 
to a wider class of metrics than those described in \cite[Theorem 5.2.1]{BZ}, 
i.e., we mean that  they assumed that sectional curvature is {\bf non-negative everywhere}.

\begin{theorem}\label{thm}
Let $(X, \partial X)$ be a complete connected Riemannian manifold $X$ with 
disconnected smooth compact convex boundary $\partial X$ 
whose radial curvature is bounded from below by $0$. 
Then, $X$ is isometric to $[0, \ell] \times \partial X_{1}$ with Euclidean product metric of 
$[0, \ell]$ and $\partial X_{1}$, where $\partial X_{1}$ denotes a connected component 
of $\partial X$. In particular, $\partial X_{1}$ is the soul of $X$. 
\end{theorem}

\medskip\noindent
Toponogov's comparison theorem for open triangles (Theorem \ref{thm4.9}) will be applied 
in the proof of Theorem \ref{thm} (see Section \ref{sec:app2}). Notice that non-negative radial curvature {\bf does not always mean} 
non-negative sectional curvature (cf.\,\cite[Example 5.6]{KT1}). 
Although Theorem \ref{thm} extends one of Burago and Zalgaller' splitting theorems 
to a wider class of metrics than those described in \cite[Theorem 5.2.1]{BZ}, 
Ichida \cite{I} and Kasue \cite{K} obtain the same conclusion of the theorem 
under weaker assumptions, i.e., the mean curvature 
(with respect to the inner normal direction) of boundary are non-negative, 
and that Ricci curvature is non-negative everywhere.

\medskip

In the following sections, all geodesics will be normalized, unless otherwise stated.  

\section{Toponogov's Theorems for Open Triangles}\label{sec:sch1}
Throughout this section, let $(X, \partial X)$ denote a complete connected Riemannian 
manifold $X$ with smooth {\bf convex} boundary $\partial X$ 
whose radial curvature is bounded from below by that of a model surface $(\wt{X}, \partial \wt{X})$ 
with its metric (\ref{model-metric}). 

\begin{definition}{\bf (Open Triangles)} 
For any fixed two points $p, q \in X \setminus \partial X$, an {\em open triangle} 
\[
{\rm OT}(\partial X, p, q) = (\partial X, p, q\,;\gamma, \mu_{1}, \mu_{2})
\]
in $X$ is defined by two $\partial X$-segments $\mu_{i} : [0, \ell_{i}] \lra X$, $i = 1, 2$, 
a minimal geodesic segment $\gamma : [0, d(p, q)] \lra X$, and $\partial X$ such that
$\mu_{1}(\ell_{1}) = \gamma (0) = p$, $\mu_{2}(\ell_{2}) = \gamma (d(p, q)) = q$.
\end{definition}

\begin{remark}
In this article, whenever an open triangle 
${\rm OT}(\partial X, p, q) = (\partial X, p, q\,;\gamma, \mu_{1}, \mu_{2})$ 
in $X$ is given, $(\partial X, p, q\,;\gamma, \mu_{1}, \mu_{2})$, as a symbol, 
always means that the minimal geodesic segment 
$\gamma$ is the opposite side to $\partial X$ emanating from $p$ to $q$, and that 
the $\partial X$-segments $\mu_{1}, \mu_{2}$ are sides emanating from $\partial X$ to $p$, $q$, 
respectively.
\end{remark}

\begin{definition}
We call the set $\wt{X} (\theta) : = \tilde{y}^{-1} ((0, \theta))$ a sector in $\wt{X}$ 
for each constant number $\theta >0$. 
\end{definition}

\begin{remark}
Since a map $(\tilde{p}, \tilde{q}) \lra (\tilde{p}, \tilde{q} + c)$, $c \in \R$, over $\wt{X}$ is an isometry, 
a sector $\wt{X} (\theta)$ is isometric to $\tilde{y}^{-1} (c, c + \theta)$ for all $c \in \R$. 
\end{remark}

\bigskip

Toponogov's comparison theorem for open triangles is stated as follows:

\begin{theorem}{\rm (\cite[Theorem 8.4]{KT2})}\label{thm4.9} Let $(X, \partial X)$ be a complete connected Riemannian manifold $X$ with smooth convex boundary $\partial X$ 
whose radial curvature is bounded from below by that of 
a model surface $(\wt{X}, \partial \wt{X})$ with its metric (\ref{model-metric}). 
Assume that 
$\wt{X}$ admits a sector $\wt{X}(\theta_{0})$ which has no pair of cut points. 
Then, for every open triangle ${\rm OT}(\partial X, p, q) = (\partial X, p, q\,;\,\gamma, \mu_{1}, \mu_{2})$ in $X$ with $d (\mu_{1}(0), \mu_{2}(0)) < \theta_{0}$,
there exists an open triangle 
${\rm OT}(\partial \wt{X}, \tilde{p}, \tilde{q}) = (\partial \wt{X}, \tilde{p}, \tilde{q}\,;\,\tilde{\gamma}, \tilde{\mu}_{1}, \tilde{\mu}_{2})$ in $\wt{X}(\theta_{0})$ 
such that
\begin{equation}\label{thm4.9-length}
d(\partial \wt{X},\tilde{p}) = d(\partial X, p), \quad 
d(\tilde{p},\tilde{q}) = d(p, q), \quad 
d(\partial \wt{X},\tilde{q}) = d(\partial X, q)
\end{equation}
and that
\begin{equation}\label{thm4.9-angle}
\angle\,p \ge \angle\,\tilde{p}, \quad  
\angle\,q \ge \angle\,\tilde{q}, \quad 
d (\mu_{1}(0), \mu_{2}(0)) \ge d (\tilde{\mu}_{1}(0), \tilde{\mu}_{2}(0)).
\end{equation}
Furthermore, if 
$d (\mu_{1}(0), \mu_{2}(0)) = d (\tilde{\mu}_{1}(0), \tilde{\mu}_{2}(0))$ holds, then 
\[
\angle\,p = \angle\,\tilde{p}, \quad  \angle\,q = \angle\,\tilde{q}
\]
hold. Here $\angle\,p$ denotes the angle between two vectors $\gamma'(0)$ and 
$-\,\mu_{1}'(d(\partial X, p))$ in $T_{p}X$.
\end{theorem}

\begin{remark}
In Theorem \ref{thm4.9}, we do not assume that $\partial X$ is connected. Moreover, 
the opposite side $\gamma$ of ${\rm OT}(\partial X, p, q)$ does not meet 
$\partial X$ (see \cite[Lemma 6.1]{KT2}). 
In \cite{MS}, they treat a pair $(M, N)$ of a complete connected Riemannian manifold $M$ 
and a compact connected totally geodesic hypersurface $N$ of $M$ such that 
the radial curvature with respect to $N$ is bounded from below by that of 
the model $((a, b) \times_{m} N, N)$, where $(a, b)$ denotes an interval, in their sense. 
Note that the radial curvature with respect to $N$ is bounded from below by that of our model 
$([0, \infty), d\tilde{x}^2) \times_{m} (\R, d\tilde{y}^2)$, 
if it is bounded from below by that of their model $((a, b) \times_{m} N, N)$. 
Thus, Theorem \ref{thm4.9} is {\bf applicable to} the pair $(M, N)$.
\end{remark}

\bigskip

In the following, we will prove the Toponogov comparison theorem for open triangles 
in a weak form (Theorem \ref{weak}), where we do not demand any assumption on a sector. To do so, 
we need to introduce definitions and a key lemma:

\begin{definition}{\bf (Generalized open triangles)} 
A generalized open triangle 
\[
{\rm GOT}(\partial \wt{X}, \wh{p}, \wh{q}\,) = (\partial \wt{X}, \wh{p}, \wh{q}\,;\,\wh{\gamma}, \wh{\mu}_{1}, \wh{\mu}_{2})
\]
in $\wt{X}$ is defined by two $\partial \wt{X}$-segments $\wh{\mu}_{i} : [0, \ell_{i}] \lra \wt{X}$, $i = 1, 2$, 
and a geodesic segment $\wh{\gamma}$ emanating from $\wh{p}$ to $\wh{q}$ such that 
$\wh{\mu}_{1}(\ell_{1}) = \wh{\gamma} (0) = \wh{p}$, 
$\wh{\mu}_{2}(\ell_{2}) = \wh{\gamma} (d(\,\wh{p},\, \wh{q}\,)) = \wh{q}$, 
and that $\wh{\gamma}$ is a shortest arc joining $\wh{p}$ to $\wh{q}$ in the compact domain 
bounded by $\wh{\mu}_{1}$, $\wh{\mu}_{2}$, and $\wh{\gamma}$.
\end{definition}

\begin{definition}{\bf (The injectivity radius)} 
The {\em injectivity radius} $\inj (\tilde{p})$ of a point $\tilde{p} \in \wt{X}$ 
is the supremum of $r > 0$ such that, for any point $\tilde{q} \in \wt{X}$ 
with $d(\tilde{p}, \tilde{q}) < r$, 
there exists a unique minimal geodesic segment joining $\tilde{p}$ to $\tilde{q}$. 
\end{definition}

\begin{remark}
For each point $\tilde{p} \in \wt{X} \setminus \partial \wt{X}$,  
$\inj(\tilde{p}) > d(\partial \wt{X}, \tilde{p})$ holds, 
if $\tilde{p}$ is sufficiently close to $\partial \wt{X}$.
\end{remark}

\begin{definition}\label{def3.3}{\bf (Thin Open Triangle)}
An open triangle ${\rm OT}(\partial X, p, q)$ in $X$ is called a 
{\em thin open triangle}, if 
\begin{enumerate}[{\rm ({TOT--}1)}]
\item
the opposite side $\gamma$ of ${\rm OT}(\partial X, p, q)$ to $\partial X$ emanating from $p$ 
to $q$ is contained in a normal convex neighborhood in $X \setminus \partial X$, and
\item
$L(\gamma) < \inj (\tilde{q}_{s})$ for all $s \in [0, d(p, q)]$,
\end{enumerate}
where $L(\gamma)$ denotes the length of $\gamma$, 
and $\tilde{q}_{s}$ denotes a point in $\wt{X}$ with 
$d(\partial \wt{X}, \tilde{q}_{s}) = d(\partial X, \gamma (s))$ for each $s \in [0, d(p, q)]$. 
\end{definition}

\medskip

Then, we have the key lemma to prove the weaker version of 
Toponogov's comparison theorem for open triangles. 

\begin{lemma}{\rm (\cite[Lemma 5.8]{KT2})}\label{lem3.8} 
For every thin open triangle ${\rm OT}(\partial X, p, q)$ in $X$, 
there exists an open triangle 
${\rm OT}(\partial \wt{X}, \tilde{p}, \tilde{q})$ in $\wt{X}$ such that
\begin{equation}\label{lem3.8-length}
d(\partial \wt{X},\tilde{p}) = d(\partial X, p), \quad 
d(\tilde{p},\tilde{q}) = d(p, q), \quad 
d(\partial \wt{X},\tilde{q}) = d(\partial X, q)
\end{equation}
and that
\begin{equation}\label{lem3.8-angle}
\angle\,p \ge \angle\,\tilde{p}, \quad  
\angle\,q \ge \angle\,\tilde{q}.
\end{equation}
\end{lemma}

\medskip

Now, the weaker version of 
Toponogov's comparison theorem for open triangles is stated as follows:

\begin{theorem}\label{weak}
Let $(X, \partial X)$ be a complete connected Riemannian 
manifold $X$ with smooth convex boundary $\partial X$ 
whose radial curvature is bounded from below by that of 
a model surface $(\wt{X}, \partial \wt{X})$. 
Then, for every open triangle ${\rm OT}(\partial X, p, q) = (\partial X, p, q\,;\,\gamma, \mu_{1}, \mu_{2})$ in $X$, 
there exists a generalized open triangle 
${\rm GOT}(\partial \wt{X}, \wh{p}, \wh{q}\,) = (\partial \wt{X}, \wh{p}, \wh{q}\,;\,\wh{\gamma}, \wh{\mu}_{1}, \wh{\mu}_{2})$ 
in $\wt{X}$ such that
\begin{equation}\label{weak-length1}
d(\partial \wt{X},\wh{p}\,) = d(\partial X, p), \quad 
d(\partial \wt{X},\wh{q}\,) = d(\partial X, q), 
\end{equation}
and
\begin{equation}\label{weak-length2}
d(\partial X, q) - d(\partial X, p) \le d(\,\wh{p},\,\wh{q}\,) \le L(\,\wh{\gamma}\,) \le d(p, q), \quad 
\end{equation}
and that
\begin{equation}\label{weak-angle}
\angle\,p \ge \angle\,\wh{p}, \quad  
\angle\,q \ge \angle\,\wh{q}.
\end{equation}
Here $L(\,\wh{\gamma}\,)$ denotes the length of $\wh{\gamma}$.
\end{theorem}

\begin{proof}
Let 
$s_{0} := 0 < s_{1} <\,\cdots\,< s_{k -1} <s_{k} :=  d(p, q)$ 
be a subdivision of $[0, d(p, q)]$ such that, 
for each $i \in \{ 1, \, \ldots\,, k \}$, the open triangle 
${\rm OT}(\partial X, \gamma (s_{i-1}), \gamma (s_{i}))$ is thin. 
It follows from Lemma \ref{lem3.8} that, for each triangle ${\rm OT}(\partial X, \gamma (s_{i-1}), \gamma (s_{i}))$, 
there exists an open triangle 
$\wt{\triangle}_{i} := {\rm OT}(\partial \wt{X}, \tilde{\gamma} (s_{i-1}), \tilde{\gamma} (s_{i}))$ in $\wt{X}$ such that 
\begin{align}
&
d(\partial \wt{X},\tilde{\gamma} (s_{i-1})) = d(\partial X, \gamma (s_{i-1})),
\label{weak-1}\\[2mm]
&
d(\tilde{\gamma} (s_{i-1}), \tilde{\gamma} (s_{i})) = d(\gamma (s_{i-1}), \gamma (s_{i})), 
\label{weak-2}\\[2mm]
&
d(\partial \wt{X}, \tilde{\gamma} (s_{i})) = d(\partial X, \gamma (s_{i})),
\label{weak-3}
\end{align}
and that 
\begin{align}
&\angle(\partial X, \gamma (s_{i-1}), \gamma (s_{i})) 
\ge 
\angle(\partial \wt{X}, \tilde{\gamma} (s_{i-1}), \tilde{\gamma} (s_{i})),\label{weak-4}\\[2mm]  
&\angle(\partial X,  \gamma (s_{i}), \gamma (s_{i-1})) 
\ge 
\angle(\partial \wt{X}, \tilde{\gamma} (s_{i}), \tilde{\gamma} (s_{i-1})).\label{weak-5}
\end{align}
Here $\angle(\partial X, \gamma (s_{i-1}), \gamma (s_{i})) $ denotes 
the angle between two sides joining $\gamma (s_{i-1})$ to $\partial X$ and 
$\gamma (s_{i})$ forming the triangle ${\rm OT}(\partial X, \gamma (s_{i-1}), \gamma (s_{i}))$.
Under this situation, 
draw $\wt{\triangle}_{1} = {\rm OT}(\partial \wt{X}, \tilde{p}, \tilde{\gamma} (s_{1}))$ in $\wt{X}$ 
satisfying (\ref{weak-1}), (\ref{weak-2}), (\ref{weak-3}), (\ref{weak-4}), (\ref{weak-5}) for $i = 1$. 
Inductively, we draw an open triangle 
$\wt{\triangle}_{i + 1} = {\rm OT}(\partial \wt{X}, \tilde{\gamma} (s_{i}), \tilde{\gamma} (s_{i + 1}))$ in $\wt{X}$, 
which is adjacent to $\wt{\triangle}_{i}$ so as to have 
the $\partial \wt{X}$-segment to $\tilde{\gamma} (s_{i})$ as a common side. 
Since 
\[
\angle(\partial X,  \gamma (s_{i}), \gamma (s_{i-1})) 
+ 
\angle(\partial X,  \gamma (s_{i}), \gamma (s_{i+1})) 
= \pi,
\]
for each $i = 1, 2,\, \ldots \,, k -1$,  we get, by (\ref{weak-4}) and (\ref{weak-5}), 
\begin{equation}\label{weak-6}
\angle(\partial \wt{X}, \tilde{\gamma} (s_{i}), \tilde{\gamma} (s_{i-1})) 
+ 
\angle(\partial \wt{X}, \tilde{\gamma} (s_{i}), \tilde{\gamma} (s_{i+1})) 
\le \pi
\end{equation}
and 
\begin{equation}\label{weak-7}
\angle\,p \ge \angle(\partial \wt{X}, \tilde{\gamma} (s_{0}), \tilde{\gamma} (s_{1})), \quad 
\angle\,q \ge \angle(\partial \wt{X}, \tilde{\gamma} (s_{k}), \tilde{\gamma} (s_{k -1})).
\end{equation}
Then, we get a domain $\cD$ bounded by two $\partial \wt{X}$-segments $\wt{\mu}_{0}, \wt{\mu}_{k}$ 
to $\tilde{\gamma} (s_{0}), \tilde{\gamma} (s_{k})$, respectively, and $\wt{\eta}$, where 
$\wt{\eta}$ denotes the broken geodesic consisting of the opposite sides 
of $\wt{\triangle}_{i}$ ($i = 1, 2, \, \ldots,\, k$) to $\partial \wt{X}$. 
Since the domain $\cD$ is locally convex by (\ref{weak-6}), 
there exists a minimal geodesic segment $\wh{\gamma}$ in the closure of $\cD$ joining 
$\tilde{\gamma} (s_{0})$ to $\tilde{\gamma} (s_{k})$. 
From (\ref{weak-7}), it is clear that the generalized open triangle 
$(\partial \wt{X}, \tilde{\gamma} (s_{0}), \tilde{\gamma} (s_{0})\,;\,\wh{\gamma}, \wt{\mu}_{0}, \wt{\mu}_{k})$ has the required properties in our theorem. 
$\qedd$
\end{proof}

\section{Definitions and notations for Sections \ref{sec:app1} and \ref{sec:app2}}

Throughout this section, let $(X, \partial X)$ denote a complete connected 
Riemannian manifold $X$ with smooth boundary $\partial X$. Our purpose of this section is to 
recall the definitions of $\partial X$-Jacobi fields, focal loci of $\partial X$, and cut loci of $\partial X$, 
which will appear in Sections \ref{sec:app1} and \ref{sec:app2}.

\begin{definition}{\bf ($\partial X$-Jacobi field)} 
Let $\mu :[0, \infty) \lra X$ be a unit speed geodesic emanating perpendicularly from $\partial X$. 
A Jacobi field $J_{\partial X}$ along $\mu$ is called a {\em $\partial X$-Jacobi field}, 
if $J_{\partial X}$ satisfies 
$J_{\partial X} (0) \in T_{\mu(0)} \partial X$ and 
$J'_{\partial X} (0) + A_{\mu'(0)} (J_{\partial X} (0)) \in (T_{\mu(0)} \partial X)^{\bot}$.
Here $J'$ denotes the covariant derivative of $J$ along $\mu$, 
and $A_{\mu'(0)}$ denotes the shape operator of $\partial X$. 
\end{definition}

\begin{definition}{\bf (Focal locus of $\partial X$)} A point $\mu(t_{0})$, $t_{0} \not= 0$, 
is called a {\em focal point of $\partial X$} along a unit speed geodesic $\mu :[0, \infty) \lra X$ 
emanating perpendicularly from $\partial X$, 
if there exists a non-zero $\partial X$-Jacobi field $J_{\partial X}$ along $\mu$ such that 
$J_{\partial X} (t_{0}) = 0$. 
The {\em focal locus $\Focal (\partial X)$ of $\partial X$} is the union of the focal points of 
$\partial X$ along all of the unit speed geodesics emanating perpendicularly 
from $\partial X$.
\end{definition}

\begin{definition}{\bf (Cut locus of $\partial X$)}\label{def_of_cut_locus2011_02_16}
Let $\mu :[0, \ell_{0}] \lra X$ be a $\partial X$-segment. 
The end point $\mu(\ell_{0})$ of $\mu ([0, \ell_{0}])$ 
is called a {\em cut point of} $\partial X$ along $\mu$, 
if any extended geodesic $\bar{\mu} :[0,\ell_{1}] \lra X$ of $\mu$, $\ell_{1} > \ell_{0}$, 
is not a $\partial X$-segment anymore. 
The {\em cut locus $\Cut (\partial X)$ of $\partial X$} is 
the union of the cut points of $\partial X$ along all 
of the $\partial X$-segments.
\end{definition}

\section{Proof of Theorem \ref{thm1.4}}\label{sec:app1}
From the similar argument in the proof of \cite[Lemma 3.1]{ST}, 
one may prove 

\begin{lemma}\label{lem6.1}
Let 
\begin{align}
&f''(t) + K(t)f(t) = 0, \quad f(0) = 1, \quad t \in [0, \infty), \notag\\[2mm]
&m''(t) + G(t)m(t) = 0, \quad m(0) = 1, \quad m'(0) = 0, \quad t \in [0, \infty),\notag
\end{align}
be two ordinary differential equations with $K(t) \ge G(t)$ on $[0, \infty)$. 
\begin{enumerate}[{\rm ({L--}1)}]
\item
If $f > 0$ on $(0, \infty)$, $f'(0) = 0$, and 
\[
\int_{0}^{\infty} \frac{1}{m (t)^{2}}dt = \infty,
\] 
then $K(t) = G(t)$ on $[0, \infty)$. 
\item
If $m > 0$ on $(0, \infty)$, $f'(0) < 0$, and 
\[
\int_{0}^{\infty} \frac{1}{m (t)^{2}}dt = \infty,
\] 
then there exists $t_{0} \in (0, \infty)$ such that $f > 0$ on $[0, t_{0})$ and $f(t_{0}) = 0$. 
\end{enumerate} 
\end{lemma}

\bigskip

Hereafter, let $(X, \partial X)$ be a complete non-compact connected Riemannian 
$n$-manifold $X$ with smooth {\bf convex} boundary $\partial X$ whose radial curvature 
is bounded from below by that of a model surface $(\wt{X}, \partial \wt{X})$ 
with its metric (\ref{model-metric}). Moreover, we denote by 
\[
\cI_{\partial X}^{\ell} (V, W) := I_{\ell} (V, W) - 
\big\langle 
A_{\mu'(0)} (V (0)), W (0)
\big\rangle 
\]
the index form with respect to a $\partial X$-segment $\mu : [0, \ell] \lra X$ for piecewise $C^{\infty}$ 
vector fields $V, W$ along $\mu$, where we set
\[
I_{\ell} (V, W) 
:= 
\int^{\ell}_{0} 
\left\{ 
\big\langle 
V', W'
\big\rangle
-
\big\langle 
R(\mu', V) \mu', W
\big\rangle
\right\}
dt,
\]
which is a symmetric bilinear form. Furthermore, 

\begin{center}
we assume that $X$ admits at least one $\partial X$-ray. 
\end{center}

\medskip\noindent 
By Lemma \ref{lem6.1}, we have 

\begin{lemma}\label{lem6.2}
Let $\mu :[0, \infty) \lra X$ be a $\partial X$-ray. 
If $(\wt{X}, \partial \wt{X})$ satisfies 
\[
\int_{0}^{\infty} \frac{1}{m (t)^{2}}dt = \infty,
\]
then, $\mu(0)$ is the geodesic point in $\partial X$, i.e., 
the second fundamental form vanishes at the point. 
\end{lemma}

\begin{proof}
Let $E$ be a unit parallel vector field along $\mu$ such that
\begin{equation}\label{lem6.2-1}
A_{\mu'(0)} (E(0)) = \lambda E(0),
\end{equation}
\begin{equation}\label{lem6.2-2}
E(t)\!\perp\!\mu'(t). 
\end{equation}
Here $\lambda$ denotes an eigenvalue of the shape operator $A_{\mu'(0)}$ of $\partial X$. 
Since $\partial X$ is convex, 
$\lambda \ge 0$ holds. 
Consider a smooth vector field $Y (t) := f(t) E(t)$ along $\mu$ satisfying 
\[
f''(t) + K_{X}(\mu'(t), E(t))f(t) = 0,
\]
with initial conditions 
\begin{equation}\label{lem6.2-4}
\quad f(0) = 1, \quad f'(0) = - \lambda. 
\end{equation}
Here $K_{X}(\mu'(t), E(t))$ denotes 
the sectional curvature with respect to the $2$-dimensional linear space spanned 
by $\mu'(t)$ and $E(t)$ at $\mu(t)$. 
Notice that $Y$ satisfies 
$Y(0) \in T_{\mu(0)} \partial X$ and $Y'(0) + A_{\mu'(0)} (Y(0)) = 0 \in (T_{\mu(0)} \partial X)^{\bot}$,
by (\ref{lem6.2-1}), (\ref{lem6.2-2}), and (\ref{lem6.2-4}). 
Suppose that 
$\lambda > 0$.
Since $f'(0) < 0$ and 
\[
\int_{0}^{\infty} \frac{1}{m (t)^{2}}dt = \infty,
\]
it follows from (L--2) in Lemma \ref{lem6.1} that there exists $t_{0} \in (0, \infty)$ 
such that $f > 0$ on $[0, t_{0})$ and  
\begin{equation}\label{lem6.2-6}
f(t_{0}) = 0,
\end{equation}
i.e., 
\begin{equation}\label{lem6.2-6.1}
Y (t) \not= 0, \quad t \in [0, t_{0})
\end{equation}
and 
$Y(t_{0}) = 0$. 
Since 
$\big\langle R(\mu'(t), Y(t)) \mu'(t), Y(t) \big\rangle 
= f(t)^{2} \big\langle R(\mu'(t), E(t)) \mu'(t), E(t) \big\rangle 
= -f''(t)f(t)$, 
we have, by (\ref{lem6.2-4}) and (\ref{lem6.2-6}), 
\begin{equation}\label{lem6.2-7}
I_{t_{0}} (Y, Y) 
= 
\int^{t_{0}}_{0} 
\frac{d}{dt} (ff')
dt 
= f(t_{0})f'(t_{0}) - f(0)f'(0) 
= \lambda. 
\end{equation}
Thus, by (\ref{lem6.2-1}),  (\ref{lem6.2-4}), and (\ref{lem6.2-7}), 
\begin{equation}\label{lem6.2-8}
\cI_{\partial X}^{t_{0}} (Y, Y) 
= I_{t_{0}} (Y, Y) - 
\big\langle 
A_{\mu'(0)} (Y (0)), Y (0)
\big\rangle
= \lambda - \lambda = 0. 
\end{equation}
On the other hand, since $\partial X$ has no focal point along $\mu$, 
for any non-zero vector field $Z$ along $\mu$ satisfying $Z(0) \in T_{\mu(0)} \partial X$ and $Z(t_{0}) = 0$, 
\begin{equation}\label{lem6.2-9}
\cI_{\partial X}^{t_{0}} (Z, Z) > 0
\end{equation}
holds (cf. Lemma 2.9 in \cite[Chapter III]{S}). 
Thus, by (\ref{lem6.2-8}) and (\ref{lem6.2-9}),  
$Y \equiv 0$ on $[0, t_{0}]$. 
This is a contradiction to (\ref{lem6.2-6.1}). 
Therefore, $\lambda = 0$, i.e., $\mu(0)$ is the geodesic point in $\partial X$. 
$\qedd$
\end{proof}

\bigskip

Here we want to go over some fundamental tools on $(\wt{X}, \partial \wt{X})$: 
A unit speed geodesic $\tilde{\gamma} : [0, a) \lra \wt{X}$ $(0 < a \le \infty)$ is expressed by 
$
\tilde{\gamma}(s) 
= (\tilde{x}(\tilde{\gamma}(s)), \tilde{y}(\tilde{\gamma}(s)))
=: (\tilde{x}(s), \tilde{y}(s))
$.
Then, there exists a non-negative constant $\nu$ depending only 
on $\tilde{\gamma}$ such that 
\begin{equation}\label{clairaut}
\nu = m(\tilde{x}(s))^{2} |\tilde{y}'(s)| 
= m(\tilde{x}(s)) \sin \angle(\tilde{\gamma}'(s), (\partial / \partial \tilde{x})_{\tilde{\gamma}(s)}).
\end{equation}
This (\ref{clairaut}) is a famous formula -- the {\em Clairaut relation}. 
The constant $\nu$ is called the {\em Clairaut constant of} $\tilde{\gamma}$. 
Remark that, by (\ref{clairaut}),  
{\em $\nu > 0$ if and only if $\tilde{\gamma}$ is 
not a $\partial \wt{X}$-ray, or its subarc}. 
Since $\tilde{\gamma}$ is unit speed, we have, by (\ref{clairaut}), 
\begin{equation}\label{geodesic-eq}
\tilde{x}'(s) = \pm \frac{\sqrt{m(\tilde{x}(s))^{2} - \nu^2}}{m(\tilde{x}(s))}.
\end{equation}
By (\ref{geodesic-eq}), we see that 
{\em $\tilde{x}'(s) = 0$ if and only if $m(\tilde{x}(s)) = \nu$}. 
Moreover, by (\ref{geodesic-eq}), we have that, 
for a unit speed geodesic $\tilde{\gamma}(s) = (\tilde{x}(s), \tilde{y}(s))$, $s_1 \le s \le s_2$, 
with the Clairaut constant $\nu$, 
\begin{equation}\label{length-eq}
s_2 - s_1 
= \phi(\tilde{x}'(s)) \int_{\tilde{x}(s_1)}^{\tilde{x}(s_2)} \frac{m(t)}{\sqrt{m(t)^{2} -\nu^2}}\,dt,
\end{equation}
if $\tilde{x}'(s) \not= 0$ on $(s_1, s_2)$. 
Here, $\phi(\tilde{x}'(s))$ denotes the sign of $\tilde{x}'(s)$. 
Furthermore, we have a lemma with respect to the length $L(\tilde{\gamma})$ of $\tilde{\gamma}$:

\begin{lemma}\label{lem6.3}
Let $\tilde{\gamma} : [0, s_{0}] \lra \wt{X} \setminus \partial \wt{X}$ denote a unit speed 
geodesic segment with Clairaut constant $\nu$. 
Then, $L(\tilde{\gamma})$ is not less than 
\begin{equation}\label{length-ineq}
t_{2} - t_{1} + \frac{\nu^{2}}{2} \int_{t_{1}}^{t_{2}} \frac{1}{m(t)\sqrt{m(t)^{2} -\nu^{2}}}\,dt,
\end{equation} 
where we set $t_{1} := \tilde{x}(0)$ and $t_{2} := \tilde{x}(s_{0})$.
\end{lemma}

\begin{proof}
We may assume that $t_{2} > t_{1}$, otherwise (\ref{length-ineq}) is non-positive. 
Let $[s_{1}, s_{2}]$ be a sub-interval of $[0, s_{0}]$ such that 
$\tilde{x}'(s) \not= 0$ on $(s_{1}, s_{2})$. 
By (\ref{length-eq}), 
\[
L(\tilde{\gamma}|_{[s_{1}, s_{2}]}) 
= 
s_{2} - s_{1} 
=
\left|
\int_{\tilde{x}(s_1)}^{\tilde{x}(s_2)} \frac{m(t)}{\sqrt{m(t)^{2} -\nu^2}}\,dt
\right|.
\]
Since $\tilde{x}'(s) \not= 0$ for all $s \in (s_{1}, s_{2})$ with $\tilde{x}(s) \in [t_{1}, t_{2}]$, 
we may choose the numbers $s_{1}$ and $s_{2}$ in such a way that 
$\tilde{x}(s_{1}) = t_{1}$ and $\tilde{x}(s_{2}) = t_{2}$
or that $\tilde{x}(s_{1}) = t_{2}$ and $\tilde{x}(s_{2}) = t_{1}$.
Thus, we see that 
\begin{equation}\label{lem6.3-1}
L(\tilde{\gamma}) \ge \int_{t_{1}}^{t_{2}} \frac{m(t)}{\sqrt{m(t)^{2} -\nu^2}}\,dt.
\end{equation}
Since 
\[
\frac{m(t)}{\sqrt{m(t)^{2} -\nu^2}} \ge 1 + \frac{\nu^{2}}{ 2 m(t) \sqrt{m(t)^{2} -\nu^2} },
\]
we have, by (\ref{lem6.3-1}), 
\[
L(\tilde{\gamma}) \ge 
t_{2} - t_{1} + \frac{\nu^{2}}{2} \int_{t_{1}}^{t_{2}} \frac{1}{m(t)\sqrt{m(t)^{2} -\nu^{2}}}\,dt.
\]
$\qedd$
\end{proof}

The next lemma is well-known in the case of the cut locus of a point (see \cite{B}), 
Although it can be proved similarly, we here give a proof of the lemma totally different from it.  

\begin{lemma}\label{lem2009-03-03}
For any $q \in \Cut (\partial X) \cap (X \setminus \partial X)$ and any $\ve > 0$, 
there exists a point in $\Cut (\partial X) \cap B_{\ve} (q)$ which admits at least two 
$\partial X$-segments.
\end{lemma}

\begin{proof}
Suppose that the cut point $q$ admits a unique $\partial X$-segment $\mu_{q}$ to $q$. 
Then, $q$ is the first focal point of $\partial X$ along $\mu_{q}$. 
For each $p \in \partial X$, we denote by $v_{p}$ 
the inward pointing unit normal vector to $\partial X$ at $p \in \partial X$. 
And let $\cU$ be a sufficiently small open neighborhood around 
$d(\partial X, q) \mu'_{q} (0)$ in the normal bundle $\cN_{\partial X}$ of $\partial X$, 
so that there exists a number $\lambda (v_{p}) \in (0, \infty)$ 
such that $\exp^{\perp}(\lambda (v_{p}) v_{p})$ is the first focal point of $\partial X$ 
for each $\lambda (v_{p}) v_{p} \in \cU$. Set $k := \liminf_{v_{p} \to \mu'_{q} (0)} \nu (v_{p})$,
where $\nu (v_{p}) := \dim \ker (d \exp^{\perp})_{\lambda (v_{p}) v_{p}}$. 
Since $\cU$ is sufficiently small, 
we may assume that 
$\nu (v_{p}) \ge k$ on $\cU_{\lambda} := \left\{ w / \| w\| \, | \, w \in \cU \right\}$, 
which is open in the unit sphere normal bundle of $\partial X$. 
It is clear that, for each integer $m \ge 0$, the set 
$
\left\{
v_{p} \in \cU_{\lambda} \, | \, \ra (d \exp^{\perp})_{\lambda (v_{p}) v_{p}} \ge m
\right\}
$
is open in $\cU_{\lambda}$. 
Hence, by \cite[Lemma 1]{IT2}, 
$\lambda$ is smooth on the open set  
$
\left\{
v_{p} \in \cU_{\lambda} \, | \, \nu (v_{p}) \le k
\right\}
= 
\left\{
v_{p} \in \cU_{\lambda} \, | \, \nu (v_{p}) = k
\right\} \subset \cU_{\lambda}$. Since 
$
(d \exp^{\perp})_{\lambda (v_{p}) v_{p}} : 
T_{\lambda (v_{p}) v_{p}} \, \cN_{\partial X} \lra T_{ \exp^{\perp} (\lambda (v_{p})v_{p})} X
$
is a linear map depending smoothly on $v_{p} \in \cU_{\lambda}$, 
there exists a {\bf non-zero} vector field $W$ on $\cU_{\lambda}$ such that 
$W_{v_{p}} \in \ker (d \exp^{\perp})_{\lambda (v_{p}) v_{p}}$ on $\cU_{\lambda}$. 
Here, we assume that 
$\ker (d \exp^{\perp})_{\lambda (v_{p}) v_{p}} \subset T_{v_{p}} \cU_{\lambda}$
by the natural identification.\par 
Assume that that there exists a sequence $\{ \mu_{i} : [0, \ell_{i}] \lra X \}$ of $\partial X$-segments 
convergent to $\mu_{q}$ such that $\mu_{i} (\ell_{i}) \in \Cut (\partial X)$ and 
$\mu_{i} (\ell_{i}) \not\in \Focal (\partial X)$ along $\mu_{i}$. 
Then it is clear that each $\mu_{i} (\ell_{i})$ admits at least two $\partial X$-segments. 
Hence, we have proved our lemma in this case.\par
Assume that 
$\exp^{\perp} (\lambda (v_{p}) v_{p}) \in \Cut (\partial X)$ for all $v_{p} \in \cU_{\lambda}$. 
Let $\sigma (s)$, $s \in (- \delta, \delta)$, be the local integral curve of $W$ on $\cU_{\lambda}$ with 
$ \mu_{q}'(0) = \sigma (0)$. 
Hence, 
$(d \exp^{\perp})_{\lambda (\sigma (s)) \sigma (s)} (\sigma' (s)) = 0$ on $(- \delta, \delta)$. 
By \cite[Lemma 1]{IT1}, 
$\exp^{\perp} (\lambda (\sigma (s)) \sigma (s)) 
= \exp^{\perp} (\lambda (\sigma (0)) \sigma (0)) = q$ holds. 
Hence $q$ is a point in $\Cut (\partial X)$ admitting at least two $\partial X$-segments.
$\qedd$
\end{proof}

\begin{remark}
Lemma \ref{lem2009-03-03} holds without curvature assumption on $(X, \partial X)$. 
\end{remark}

\begin{proposition}\label{prop6.4}
Let $\mu_{0} :[0, \infty) \lra X$ be a $\partial X$-ray guaranteed by the assumption above. 
If $(\wt{X}, \partial \wt{X})$ satisfies 
\begin{equation}\label{prop6.4-con1}
\int_{0}^{\infty} \frac{1}{m (t)^{2}}dt = \infty,
\end{equation}
or 
\begin{equation}\label{prop6.4-con2}
\liminf_{t \to \infty} m(t) = 0,
\end{equation}
then, any point of $X$ lies in a unique $\partial X$-ray. 
In particular, $\partial X$ is totally geodesic in the case where (\ref{prop6.4-con1}) is satisfied.
\end{proposition}

\begin{proof}
Choose any point $q \in X \setminus \partial X$ not lying on $\mu_{0}$. 
Let $\mu_{1} : [0, d(\partial X, q)] \lra X$ denote a $\partial X$-segment with 
$\mu_{1} (d(\partial X, q)) = q$. 
For each $t > 0$, let $\gamma_{t} :[0, d(q, \mu_{0}(t))]\lra X$ denote a minimal geodesic segment 
emanating from $q$ to $\mu_{0} (t)$. 
From Theorem \ref{weak} and the triangle inequality, 
it follows that there exists a generalized open triangle 
\[
{\rm GOT}(\partial \wt{X}, \wh{\mu}_{0}(t), \wh{q}\,) 
= 
(\partial \wt{X}, \wh{\mu}_{0}(t), \wh{q}\,;\,\wh{\gamma}_{t}, \wh{\mu}_{0}^{(t)}, \wh{\mu}_{1})
\] 
in $\wt{X}$ corresponding to the triangle 
${\rm OT}(\partial X, \mu_{0}(t), q) = (\partial X, \mu_{0}(t), q\,;\,\gamma_{t}, \mu_{0}|_{[0, \,t]}, \mu_{1})$ in $X$ such that 
\begin{equation}\label{prop6.4-length1}
d(\partial \wt{X}, \wh{\mu}_{0}(t)) = t, \quad 
d(\partial \wt{X},\wh{q}\,) = d(\partial X, q), 
\end{equation}
and
\begin{equation}\label{prop6.4-length2}
L(\,\wh{\gamma}_{t}\,) \le d(\mu_{0}(t), q) \le t + d(q, \mu_{0}(0)) 
\end{equation}
and that
\begin{equation}\label{prop6.4-angle}
\angle(\partial X, q, \mu_{0} (t)) 
\ge 
\angle(\partial \wt{X}, \wh{q}, \wh{\mu}_{0} (t)).
\end{equation}
Here $\angle(\partial X, q, \mu_{0} (t))$ denotes 
the angle between two sides $\mu_{1}$ and $\gamma_{t}$ joining $q$ 
to $\partial X$ and $\mu_{0} (t)$ forming the triangle ${\rm OT}(\partial X, \mu_{0}(t), q)$. 
From Lemma \ref{lem6.3}, (\ref{prop6.4-length1}), and (\ref{prop6.4-length2}), 
we get 
\begin{align}\label{prop6.4-1}
t + d(q, \mu_{0}(0)) 
&\ge 
L(\,\wh{\gamma}_{t}\,)\notag\\[2mm] 
&\ge 
t - d(\partial X, q) + \frac{\nu_{t}^{2}}{2} \int_{d(\partial X,\, q)}^{t} \frac{1}{m(t)\sqrt{m(t)^{2} -\nu_{t}^{2}}}\,dt.
\end{align}
where $\nu_{t}$ denotes the Clairaut constant of $\wh{\gamma}_{t}$. 
By (\ref{prop6.4-1}), 
\begin{equation}\label{prop6.4-2}
d(\partial X, q) + d(q, \mu_{0}(0)) 
\ge 
\frac{\nu_{t}^{2}}{2} \int_{d(\partial X,\, q)}^{t} \frac{1}{m(t)^{2}}\,dt.
\end{equation}\par
First, assume that $(\wt{X}, \partial \wt{X})$ satisfies (\ref{prop6.4-con1}). 
Then, it is clear from (\ref{prop6.4-2}) that 
$\lim_{t \to \infty} \nu_{t} = 0$. Hence, by (\ref{clairaut}), we have 
\begin{equation}\label{prop6.4-3}
\lim_{t \to \infty} \angle(\partial \wt{X}, \wh{q}, \wh{\mu}_{0} (t)) = \pi.
\end{equation}
By (\ref{prop6.4-angle}) and (\ref{prop6.4-3}), 
$\gamma_{\infty} := \lim_{t \to \infty} \gamma_{t}$ is a ray emanating from $q$ such that 
\[
\angle\,(\gamma'_{\infty}(0), - \mu'_{1}(d(\partial X, q))) = \pi.
\]
This implies that $q$ lies on a unique $\partial X$-segment. 
Therefore, by Lemma \ref{lem2009-03-03}, $q$ lies on a $\partial X$-ray. 
Now, it is clear from Lemma \ref{lem6.2} that $\partial X$ is totally geodesic.\par
Second, assume that $(\wt{X}, \partial \wt{X})$ satisfies (\ref{prop6.4-con2}). 
Then, there exists a divergent sequence $\{ t_{i} \}_{i \in \N}$ such that 
\begin{equation}\label{prop6.4-4}
\lim_{t \to \infty} m(t_{i}) = 0.
\end{equation}
From (\ref{clairaut}), we see 
\begin{equation}\label{prop6.4-5}
\nu_{i} \le m (t_{i}),
\end{equation}
where $\nu_{i}$ denotes the Clairaut constant of $\wh{\gamma}_{t_{i}}$. 
Hence, by (\ref{prop6.4-4}) and (\ref{prop6.4-5}), 
$\liminf_{t \to \infty} \nu_{t} = 0$ holds. Now, 
it is clear that there exist a limit geodesic $\gamma_{\infty}$ of $\{\gamma_{t_{i}}\}$ 
such that $\gamma_{\infty}$ is a ray emanating from $q$ and satisfies 
$\angle\,(\gamma'_{\infty}(0), - \mu'_{1}(d(\partial X, q))) = \pi$. 
Therefore, by Lemma \ref{lem2009-03-03}, $q$ lies on a $\partial X$-ray.
$\qedd$
\end{proof}

\bigskip

By Proposition \ref{prop6.4}, there does not exist a cut point of $\partial X$. 
Therefore, it is clear that 

\begin{corollary}\label{cor6.4}
If $(\wt{X}, \partial \wt{X})$ satisfies (\ref{prop6.4-con1}), or (\ref{prop6.4-con2}), 
then $X$ is diffeomorphic to $[0, \infty) \times \partial X$.
\end{corollary}

\bigskip

Furthermore, we may reach stronger conclusion 
than Corollary \ref{cor6.4}\,:

\begin{theorem}\label{prop6.5}
If $(\wt{X}, \partial \wt{X})$ satisfies 
\[
\int_{0}^{\infty} \frac{1}{m (t)^{2}}dt = \infty,
\]
then, for every $\partial X$-ray $\mu : [0, \infty) \lra X$, 
the radial curvature $K_{X}$ satisfies
\begin{equation}\label{prop6.5-con2}
K_{X}(\sigma_{t}) = G (\tilde{\mu}(t))
\end{equation}
for all $t \in [0, \infty)$ and all $2$-dimensional linear space $\sigma_{t}$ spanned by $\mu'(t)$ 
and a tangent vector to $X$ at $\mu(t)$. 
In particular, $X$ is isometric to the warped product manifold 
$[0, \infty ) \times_{m} \partial X$ of $[0, \infty )$ and 
$(\partial X, g_{\partial X})$ with the warping function $m$. 
Here $g_{\partial X}$ denotes the induced Riemannian metric from $X$. 
\end{theorem}

\begin{proof}
Take any point $p \in \partial X$, and fix it. 
By Proposition \ref{prop6.4}, 
we may take a $\partial X$-ray $\mu :[0, \infty) \lra X$ emanating from $p = \mu(0)$. 
Suppose that 
\begin{equation}\label{prop6.5-2009-01-18-1}
K_{X}(\sigma_{t_{0}}) > G (\tilde{\mu}(t_{0}))
\end{equation}
for some linear plane $\sigma_{t_{0}}$ spanned by $\mu'(t_{0})$ 
and a unit tangent vector $v_{0}$ orthogonal to $\mu'(t_{0})$. 
If we denote by $E(t)$ the parallel vector field along $\mu$ satisfying $E(t_{0}) = v_{0}$, 
then $E(t)$ is unit and orthogonal to $\mu'(t_{0})$ for each $t$. 
We define a non-zero vector field $Y(t)$ along $\mu$ by 
$Y(t) := f(t) E(t)$, 
where $f$ is the solution of the following differential equation 
\begin{equation}\label{prop6.5-2009-01-18-2}
f''(t) + K_{X} (\mu'(t), E(t)) f(t) = 0
\end{equation}
with initial condition $f(0) = 1$ and $f'(0) = 0$. 
Here $K_{X} (\mu'(t), E(t))$ denotes the sectional curvature of the plane 
spanned by $\mu'(t)$ and $E(t)$. 
It follows from (\ref{prop6.5-2009-01-18-1}) and (L--1) in Lemma \ref{lem6.1} that 
there exists $t_{1} > 0$ such that 
$f(t_{1}) = 0$. 
From (\ref{prop6.5-2009-01-18-2}), 
we get 
\begin{equation}\label{prop6.5-2009-01-18-3}
I_{t_{1}} (Y, Y) 
= 
\int^{t_{1}}_{0} 
\frac{d}{dt} (ff')
dt 
= 
0.
\end{equation}
Since $\partial X$ is totally geodesic by Proposition \ref{prop6.4}, $A_{\mu'(0)} (E(0)) = 0$. 
Thus, by (\ref{prop6.5-2009-01-18-3}), 
$\cI_{\partial X}^{t_{1}} (Y, Y) = 0$ holds. 
On the other hand, 
$\cI_{\partial X}^{t_{1}} (Y, Y) > 0$ holds, 
since there is no focal point of $\partial X$ along $\mu$. 
This is a contradiction. 
Therefore, we get the first assertion (\ref{prop6.5-con2}).\par
Now it is clear that the map $\varphi : [0, \infty ) \times_{m} \partial X \lra X$ 
defined by 
$\varphi (t, q) := \exp^{\perp} (t v_{q})$ 
gives an isometry from $[0, \infty ) \times_{m} \partial X$ onto $X$. 
Here $v_{q}$ denotes the inward pointing unit normal vector to $\partial X$ at $q \in \partial X$.
$\qedd$
\end{proof}

\section{Proof of Theorem \ref{thm}}\label{sec:app2}
Throughout this section, 
let $(X, \partial X)$ be a complete connected Riemannian manifold $X$ 
with {\bf disconnected} smooth compact {\bf convex} boundary $\partial X$ 
whose radial curvature is bounded from below by $0$. 
Under the hypothesis, we may assume
\[
\partial X = \bigcup_{i = 1}^{k} \partial X_{i}, \quad k \ge 2.
\]
Here each $\partial X_{i}$ denotes a connected component of $\partial X$ and is compact. 
Set 
\[
\ell := \min \{d(\partial X_{i}, \partial X_{j}) \,|\, 1 \le i, j \le k, i \not= j \}.
\] 
Then let $\partial X_{1}, \partial X_{2}$ denote the connected components of $\partial X$ satisfying 
\[
d(\partial X_{1}, \partial X_{2}) = \ell.
\]

The proof of the next lemma is standard: 

\begin{lemma}\label{lem8.1}
Let $\mu$ denote a minimal geodesic segment in $X$ emanating from $\partial X_{1}$ to 
$\partial X_{2}$. 
Then, there does not exist any other $\partial X$-segment 
to $\mu(\ell / 2)$ than $\mu|_{[0,\,\ell / 2]}$ and $\mu|_{[\ell / 2,\,\ell]}$. 
Furthermore, each midpoint $\mu(\ell / 2)$ is not a focal point of $\partial X$ 
along $\mu$. 
\end{lemma}

\bigskip

Hereafter, the half plane
\[
\R^{2}_{+} := \{\tilde{p} \in \R^{2} \,|\, \tilde{x}(\tilde{p}) \ge 0 \}
\]
with Euclidean metric $d\tilde{x}^{2} + d\tilde{y}^{2}$ will be used as 
the model surface for $(X, \partial X)$. 

\begin{lemma}\label{lem8.2}
Any point in $X$ lies on a minimal geodesic segment emanating 
from $\partial X_{1}$ to $\partial X_{2}$ of length $\ell$. 
In particular, $\partial X$ consists of $\partial X_{1}$ and $\partial X_{2}$.
\end{lemma}

\begin{proof}
Since $X$ is connected, it is sufficient to prove that the subset $\cO$ of $X$ is open and closed, 
where $\cO$ denotes the set of all points $r \in X$ which lies on a minimal geodesic segment 
emanating from $\partial X_{1}$ to $\partial X_{2}$ of length $\ell$. 
Since it is trivial that $\cO$ is closed, we will prove that $\cO$ is open.\par
Choose any point 
$r \in \cO$, and fix it. 
Thus, $r$ lies on a minimal geodesic segment $\mu_{1} : [0, \ell] \lra X$ 
emanating from $\partial X_{1}$ to $\partial X_{2}$. 
Set $p := \mu_{1} (\ell / 2)$. 
Let $S$ be the equidistant set from $\partial X_{1}$ and $\partial X_{2}$, i.e., 
\begin{equation}\label{lem8.2-2009-01-17}
S := \{ q \in X \,|\, d(\partial X_{1}, q) = d(\partial X_{2}, q) \}.
\end{equation}
It follows from Lemma \ref{lem8.1} that 
$S \cap B_{\ve_{1}} (p) \subset \Cut (\partial X)$, 
if $\ve_{1} > 0$ is chosen sufficiently small. Choose any point 
$q \in S \cap B_{\ve_{1}} (p) \setminus \{ p \}$, and also fix it. 
Let $\eta_{i}$, $i = 1, 2$, denote a $\partial X$-segment to $q$ such that 
$\eta_{1}(0) \in \partial X_{1}$ and $\eta_{2}(0) \in \partial X_{2}$, respectively. 
Moreover, let 
$\gamma : [0, d(p, q)] \lra X$ denote a minimal geodesic segment emanating from $p$ to $q$. 
Since 
\[
\angle (\gamma'(0), - \mu'_{1}(\ell / 2)) + \angle (\gamma'(0), \mu'_{1}(\ell / 2)) = \pi,
\]
we may assume, without loss of generality, that 
\begin{equation}\label{lem8.2-1}
\angle (\gamma'(0), - \mu'_{1}(\ell / 2)) \le \pi / 2.
\end{equation}
It follows from Theorem \ref{thm4.9} that there exists an open triangle 
\[
{\rm OT}(\partial \, \R_{+}^{2}, \tilde{p}, \tilde{q}) 
= (\partial \, \R_{+}^{2}, \tilde{p}, \tilde{q}\,;\,\tilde{\gamma}, \tilde{\mu}_{1}, \tilde{\eta}_{1})
\] 
in $\R^{2}_{+}$ corresponding to the triangle 
${\rm OT}(\partial X_{1}, p, q) = (\partial X_{1}, p, q\,;\,\gamma, \mu_{1}|_{[0,\,\ell / 2]}, \eta_{1})$ 
such that  
\begin{equation}\label{lem8.2-2}
d(\partial \, \R_{+}^{2},\tilde{p}) = \ell / 2, \quad 
d(\tilde{p},\tilde{q}) = d(p, q), \quad 
d(\partial \, \R_{+}^{2},\tilde{q}) = d(\partial X_{1}, q),
\end{equation}
and 
\begin{equation}\label{lem8.2-3}
\angle (\gamma'(0), - \mu'_{1}(\ell / 2)) = \angle\,p \ge \angle\,\tilde{p}, \quad  
\angle\,q \ge \angle\,\tilde{q}.
\end{equation} 
By (\ref{lem8.2-1}) and $\angle\,p \ge \angle\,\tilde{p}$ of (\ref{lem8.2-3}), 
we have 
\begin{equation}\label{lem8.2-4}
\angle\,\tilde{p} \le \pi / 2.
\end{equation}
Since our model is $\R^{2}_{+}$, 
it follows from the two equations $d(\partial \wt{X},\tilde{p}) = \ell / 2$, 
$d(\partial \, \R_{+}^{2},\tilde{q}) = d(\partial X_{1}, q)$ of (\ref{lem8.2-2}),  
and (\ref{lem8.2-4}) that 
\begin{equation}\label{lem8.2-5}
d(\partial X_{1}, q) = d(\partial \, \R_{+}^{2},\tilde{q}) \le \ell / 2.
\end{equation}
On the other hand, the broken geodesic segment defined by combining 
$\eta_{1}$ and $\eta_{2}$ is a curve joining $\partial X_{1}$ to $\partial X_{2}$. 
This implies that length of the broken geodesic segment is not less than that of $\mu_{1}$. 
Thus, 
\begin{equation}\label{lem8.2-6}
2 L(\eta_{1}) = L(\eta_{1}) + L(\eta_{2}) \ge \ell, 
\end{equation}
where $L(\,\cdot\,)$ denotes the length of a curve. 
Since $L(\eta_{1}) = d(\partial X_{1}, q)$, 
we have, by (\ref{lem8.2-6}), that 
\begin{equation}\label{lem8.2-7}
d(\partial X_{1}, q) \ge \ell / 2.
\end{equation}
By (\ref{lem8.2-5}) and (\ref{lem8.2-7}), 
$d(\partial X_{1}, q) =  d(\partial X_{2}, q) = \ell /2$.
Therefore, we have proved that any point $q \in S \cap B_{\ve_{1}} (p)$ is 
the midpoint of a minimal geodesic segment emanating 
from $\partial X_{1}$ to $\partial X_{2}$ of length $\ell$. 
Furthermore, by Lemma \ref{lem8.1}, each point of $S \cap B_{\ve_{1}} (p)$ is not a focal point of $\partial X$. 
It is therefore clear that any point sufficiently close to the point $r \in \cO$ is a point of $\cO$, 
i.e, $\cO$ is open.
$\qedd$
\end{proof}

\begin{remark}\label{rem8.3}
From Lemmas \ref{lem8.1} and \ref{lem8.2}, 
it is clear that 
\begin{equation}\label{sec8-1}
\Cut (\partial X) = \{ p \in X \,|\, d(\partial X, p) = \ell / 2\} = S
\end{equation}
and that 
\begin{equation}\label{sec8-2}
d(\partial X, p) \le \ell / 2
\end{equation}
for all $p \in X$. 
Here $S$ is the equidistant set defined by (\ref{lem8.2-2009-01-17}). 
Thus, from the proof of Lemma \ref{lem8.2}, we see that 
$\angle\,p = \angle\,q = \pi / 2$
holds for all $p, q \in \Cut (\partial X)$. 
\end{remark}

\begin{lemma}\label{lem8.4}
$\Cut (\partial X)$ is totally geodesic.
\end{lemma} 

\begin{proof}
Let $p$, $q$ be any mutually distinct points of $\Cut (\partial X)$, and fix them. 
Moreover, let $\gamma : [0, d(p, q)] \lra X$ denote a minimal geodesic segment emanating from 
$p$ and $q$. 
If we prove that 
$\gamma (t) \in \Cut (\partial X)$ 
for all $t \in [0, d(p, q)]$, then our proof is complete.\par 
Suppose that 
\begin{equation}\label{lem8.4-0}
\gamma (t_{0}) \not\in \Cut (\partial X)
\end{equation}
for some $t_{0} \in  (0, d(p, q))$. 
By (\ref{sec8-1}), we have that 
\begin{equation}\label{lem8.4-1}
d (\partial X, \gamma (t_{0}) ) \not= \ell / 2,
\end{equation}
and that 
\begin{equation}\label{lem8.4-2}
d (\partial X, p) = d (\partial X, q ) = \ell / 2.
\end{equation}
The equations (\ref{sec8-2}) and (\ref{lem8.4-1}) imply that 
\begin{equation}\label{lem8.4-3}
d (\partial X, \gamma (t_{0}) ) < \ell / 2.
\end{equation}
Without loss of generality, we may assume that 
\begin{equation}\label{lem8.4-4}
d (\partial X, \gamma (t_{0}) ) = \min \{d (\partial X, \gamma (t)) \,|\, 0 \le t \le d(p, q) \}.
\end{equation}
By Remark \ref{rem8.3}, (\ref{lem8.4-0}), and (\ref{lem8.4-4}), 
we obtain the open triangle ${\rm OT}(\partial X, p, \gamma (t_{0}))$ satisfying 
\begin{equation}\label{lem8.4-5}
\angle\,p = \pi / 2, \quad \angle\,\gamma (t_{0}) = \pi / 2.
\end{equation}
From Theorem \ref{thm4.9}, (\ref{lem8.4-2}), (\ref{lem8.4-3}), and (\ref{lem8.4-5}), 
we thus get an open triangle 
${\rm OT}(\partial \, \R_{+}^{2}, \tilde{p}, \tilde{\gamma} (t_{0}))$ in $\R_{+}^{2}$ 
corresponding to the triangle ${\rm OT}(\partial X, p, \gamma (t_{0}))$ such that 
\[
d (\partial \, \R_{+}^{2},\tilde{p} ) = \ell / 2, \quad 
d (\partial \, \R_{+}^{2},\tilde{\gamma} (t_{0})) < \ell / 2,  
\]
and that 
\[
\angle\,\tilde{p} \le \pi / 2, \quad
\angle\,\tilde{\gamma} (t_{0}) \le \pi / 2 .
\]
This is a contradiction, since our model is $\R^{2}_{+}$. 
Therefore, $\gamma (t) \in \Cut (\partial X)$ holds for all $t \in [0, d(p, q)]$.
$\qedd$
\end{proof}

\begin{lemma}\label{lem8.5}
For each $t \in (0, \ell / 2)$, the level set 
$H_{i} (t) := \{ p \in X \,|\, d(\partial X_{i}, p) = t \}$, $i = 1, 2$, 
is totally geodesic, and $H_{1}(t)$ is totally geodesic for all $t \in (0, \ell)$. 
\end{lemma}

\begin{proof}
Take any $t \in (0, \ell / 2)$, and fix it. 
Let $p, q$ be any mutually distinct points in $H_{1} (t)$, and also fix them. 
Let $\mu_{1}, \mu_{2} : [0, \ell] \lra X$ denote minimal geodesic segment emanating 
from $\partial X_{1}$ to $\partial X_{2}$ and passing through 
$\mu_{1} (t) = p$, $\mu_{2}(t) = q$, respectively. 
Thus, we have an open triangle 
$
{\rm OT}(\partial X_{1}, p, q) = (\partial X_{1}, p, q\,;\,\gamma_{t}, \mu_{1}|_{[0,\,t]}, \mu_{2}|_{[0,\,t]})
$, 
where $\gamma_{t} : [0, d(p, q)] \lra X$ denotes a minimal geodesic segment emanating from 
$p$ to $q$. 
If we prove 
\begin{equation}\label{lem8.5-1}
\angle\,p = \angle\,q = \pi / 2,
\end{equation}
then we see, by similar argument in the proof of Lemma \ref{lem8.4}, 
that $H_{1} (t)$ is totally geodesic. 
Thus, we will prove (\ref{lem8.5-1}) in the following.\par 
By Theorem \ref{thm4.9}, 
there exists an open triangle 
\[
{\rm OT}(\partial\,\R_{+}^{2}, \tilde{p}, \tilde{q}) 
= 
(\partial\,\R_{+}^{2}, \tilde{p}, \tilde{q}\,;\,\tilde{\gamma}_{t}, \tilde{\mu}_{1}|_{[0,\,t]}, \tilde{\mu}_{2}|_{[0,\,t]})
\]
in $\R_{+}^{2}$ corresponding to the triangle ${\rm OT}(\partial X_{1}, p, q)$ such that  
\begin{equation}\label{lem8.5-2}
d(\partial \, \R_{+}^{2},\tilde{p}) = d(\partial \, \R_{+}^{2},\tilde{q}) = t, \quad 
d(\tilde{p},\tilde{q}) = d(p, q)
\end{equation}
and that 
\begin{equation}\label{lem8.5-3}
\angle\,p \ge \angle\,\tilde{p}, \quad 
\angle\,q \ge \angle\,\tilde{q}.
\end{equation}
Since our model is $\R^{2}_{+}$, 
the equation $d(\partial \, \R_{+}^{2},\tilde{p}) = d(\partial \, \R_{+}^{2},\tilde{q})$ of (\ref{lem8.5-2}) 
implies that 
\begin{equation}\label{lem8.5-4}
\angle\,\tilde{p}
= 
\angle\,\tilde{q}
= \pi / 2.
\end{equation}
Thus, by (\ref{lem8.5-3}) and (\ref{lem8.5-4}), we have 
\begin{equation}\label{lem8.5-5}
\angle\,p \ge \pi / 2 ,\quad 
\angle\,q \ge \pi / 2 .
\end{equation}
On the other hand, by Lemma \ref{lem8.4}, $\Cut (\partial X)$ is totally geodesic, i.e., all eigenvalues of the shape operator of $\Cut (\partial X)$ are $0$ in the vector normal to $\Cut (\partial X)$. 
Since the radial vector of any $\Cut (\partial X)$-segment is parallel to that of a $\partial X$-segment, 
$\Cut (\partial X)$ has also non-negative radial curvature.
Therefore, we can apply Theorem \ref{thm4.9} to the open triangle 
\[
{\rm OT}(\Cut (\partial X), p, q) = (\Cut (\partial X), p, q\,;\,\gamma_{t}, \mu_{1}|_{[t,\,\ell/ 2]}, \mu_{2}|_{[t,\,\ell / 2]}).
\] 
Thus, by Theorem \ref{thm4.9}, 
there exists an open triangle 
\[
{\rm OT}(\partial\,\R_{+}^{2}, \wh{p}, \wh{q}\,) 
= 
(\partial\,\R_{+}^{2}, \wh{p}, \wh{q}\,;\,
\tilde{\gamma}_{t}, \tilde{\mu}_{1}|_{[t,\,\ell/ 2]}, \tilde{\mu}_{2}|_{[t,\,\ell/ 2]})
\] 
in $\R_{+}^{2}$ corresponding to the triangle ${\rm OT}(\Cut (\partial X), p, q)$ such that  
\begin{equation}\label{lem8.5-6}
d(\partial \, \R_{+}^{2},\wh{p}\,) = d(\partial \, \R_{+}^{2},\wh{q}\,) = \ell / 2 - t, \quad 
d(\,\wh{p},\,\wh{q}\,) = d(p, q)
\end{equation}
and that 
\begin{equation}\label{lem8.5-7}
\pi - \angle\,p \ge \angle\,\wh{p}, \quad 
\pi - \angle\,q \ge \angle\,\wh{q}.
\end{equation}
As well as above, the equations (\ref{lem8.5-6}) and (\ref{lem8.5-7}) imply 
$\pi - \angle\,p \ge \pi / 2$ and $\pi - \angle\,q \ge \pi / 2$,
since our model is $\R_{+}^{2}$. 
Thus, we have 
\begin{equation}\label{lem8.5-8}
\angle\,p \le \pi / 2, \quad 
\angle\,q \le \pi / 2.
\end{equation}
By (\ref{lem8.5-5}) and (\ref{lem8.5-8}), 
we therefore get (\ref{lem8.5-1}). 
By the same argument above, 
one may prove that $H_{2} (t)$ is also totally geodesic for all $t \in (0, \ell / 2)$. 
Since $H_{1} (t) = H_{2} (\ell - t)$, 
$H_{1} (t)$ is totally geodesic for all $t \in (0, \ell)$.
$\qedd$
\end{proof}

\begin{theorem}\label{thm2009-03-27}
Let $(X, \partial X)$ be a complete connected Riemannian manifold $X$ with 
disconnected smooth compact convex boundary $\partial X$ 
whose radial curvature is bounded from below by $0$. 
Then, $X$ is isometric to $[0, \ell] \times \partial X_{1}$ with Euclidean product metric of 
$[0, \ell]$ and $\partial X_{1}$, where $\partial X_{1}$ denotes a connected component 
of $\partial X$. In particular, $\partial X_{1}$ is the soul of $X$. 
\end{theorem}

\begin{proof}
Let $\Phi : [0, \ell] \times \partial X_{1} \lra X$ denote the map defined by 
$\Phi (t, p) := \exp^{\perp} (t \, v_{p})$, 
where $v_{p}$ denotes the inward pointing unit normal vector 
to $\partial X_{1}$ at $p \in \partial X_{1}$. 
We will prove that the $\Phi$ is an isometry. 
From Lemma \ref{lem8.2}, it is clear that $\Phi$ is a diffeomorphism.\par
Let $\mu_{1} : [0, \ell] \lra X$ denote any minimal geodesic segment 
emanating from $\partial X_{1}$ to $\partial X_{2}$, and fix it. 
Choose a minimal geodesic segment $\mu_{2} : [0, \ell] \lra X$ 
emanating from $\partial X_{1}$ to $\partial X_{2}$ sufficiently close $\mu_{1}$, 
so that, for each $t \in (0, \ell)$, 
$\mu_{1}(t)$ is joined with $\mu_{2}(t)$ by a unique minimal geodesic segment $\gamma_{t}$. 
Since each level hypersurface $H_{1}(t)$ is totally geodesic by Lemma \ref{lem8.5}, 
$\gamma_{t}$ meets $\mu_{1}$ and $\mu_{2}$ perpendicularly 
at $\mu_{1}(t)$ and $\mu_{2}(t)$, respectively. 
Therefore, by the first variation formula, 
\[
\frac{d}{dt} d(\mu_{1}(t), \mu_{2}(t)) = 0,
\]
holds for all $t \in (0, \ell)$. 
Thus, 
$d(\mu_{1}(t), \mu_{2}(t)) = d(\mu_{1}(0), \mu_{2}(0))$ 
holds for all $t \in [0, \ell]$. 
This implies that 
\begin{equation}\label{prop8.6-1}
\left\| 
d \Phi_{(t, \, p)} 
\left(
\frac{\partial}{\partial x_{i}}
\right)
\right\|
= 
\left\| 
d \Phi_{(0, \, p)} 
\left(
\frac{\partial}{\partial x_{i}}
\right)
\right\|
\end{equation}
for all $t \in [0, \ell]$. 
Here $(x_{1}, x_{2}, \ldots, x_{n - 1})$ denotes a system of local coordinates 
around $p := \mu_{1}(0)$ with respect to $\partial X_{1}$.
Since 
\[
d \Phi_{(0, \, p)} 
\left(
\frac{\partial}{\partial x_{i}}
\right)
= 
\left(
\frac{\partial}{\partial x_{i}}
\right)_{(0,\,p)},
\]
we get, by (\ref{prop8.6-1}), 
\begin{equation}\label{prop8.6-2}
\left\| 
d \Phi_{(t, \, p)} 
\left(
\frac{\partial}{\partial x_{i}}
\right)
\right\|
= 
\left\| 
\left(
\frac{\partial}{\partial x_{i}}
\right)_{(0,\,p)}
\right\|
= 
\left\| 
\left(
\frac{\partial}{\partial x_{i}}
\right)_{p}
\right\|.
\end{equation}
It is clear that 
\begin{equation}\label{prop8.6-3}
d \Phi_{(t, \, p)} 
\left(
\frac{\partial}{\partial x_{i}}
\right) 
\perp 
d \Phi_{(t, \, p)} 
\left(
\frac{\partial}{\partial x_{0}}
\right), \quad i = 1, 2, \ldots, n -1, 
\end{equation}
and
\begin{equation}\label{prop8.6-4}
\left\| 
d \Phi_{(t, \, p)} 
\left(
\frac{\partial}{\partial x_{0}}
\right)
\right\|
= 
1
\end{equation}
for all $t \in [0, \ell]$. 
Here $x_{0}$ denotes the standard local coordinate system for $[0, \ell]$. 
By (\ref{prop8.6-2}), (\ref{prop8.6-3}), (\ref{prop8.6-4}), 
$\Phi$ is an isometry. 
$\qedd$
\end{proof}

\medskip

\begin{center}
Kei KONDO $\cdot$ Minoru TANAKA 

\bigskip
Department of Mathematics\\
Tokai University\\
Hiratsuka City, Kanagawa Pref.\\ 
259\,--\,1292 Japan

\bigskip

{\small
$\bullet$\,our e-mail addresses\,$\bullet$

\bigskip 
\textit{e-mail of Kondo} \,:

\medskip
{\tt keikondo@keyaki.cc.u-tokai.ac.jp}

\medskip
\textit{e-mail of Tanaka}\,:

\medskip
{\tt m-tanaka@sm.u-tokai.ac.jp}
}
\end{center}

\end{document}